\documentclass[11pt]{article}

\usepackage[font=small,labelfont=bf]{caption}

\usepackage[utf8]{inputenc}
\usepackage[a4paper, total={8.5in, 11in},margin=1in]{geometry}
\usepackage{amsmath}
\usepackage{amssymb}
\usepackage{amsthm}
\usepackage{mathtools}
\usepackage{thmtools}
\usepackage{placeins}
\usepackage{array}

\usepackage{hyperref}
\hypersetup{
  linkcolor  = red,
  citecolor  = blue,
  colorlinks = true
}
\usepackage{nameref,cleveref}
\usepackage{fullpage}
\usepackage{stackengine}
\usepackage[ruled,vlined]{algorithm2e}

\newtheorem{theorem}{Theorem}[section]

\newtheorem{proposition}{Proposition}[section]

\DeclareMathOperator{\adj}{adj}

\makeatletter
\def\env@dmatrix{\hskip -\arraycolsep
  \let\@ifnextchar\new@ifnextchar
  \extrarowheight=2ex
  \array{*\c@MaxMatrixCols{>{\displaystyle}c}}}

\makeatother
\author{
  Ioannis Anagnostides\\
  \texttt{ianagnost@student.ethz.ch}
  \and
  Paolo Penna\\
  \texttt{paolo.penna@inf.ethz.ch}
}
\date{}
\title{A Robust Framework for Analyzing Gradient-Based Dynamics in Bilinear Games}

\begin{document}

\maketitle
\pagenumbering{gobble}

\begin{abstract}
    In this work, we establish a frequency-domain framework for analyzing gradient-based algorithms in linear minimax optimization problems; specifically, our approach is based on the Z-transform, a powerful tool applied in Control Theory and Signal Processing in order to characterize linear discrete-time systems. We employ our framework to obtain the first tight analysis of stability of Optimistic Gradient Descent/Ascent (OGDA), a natural variant of Gradient Descent/Ascent that was shown to exhibit last-iterate convergence in bilinear games by Daskalakis et al. \cite{DBLP:journals/corr/abs-1711-00141}. Importantly, our analysis is considerably simpler and more concise than the existing ones. 
    
    Moreover, building on the intuition of OGDA, we consider a general family of gradient-based algorithms that augment the memory of the optimization through multiple historical steps. We reduce the convergence -- to a saddle-point -- of the dynamics in bilinear games to the stability of a polynomial, for which efficient algorithmic schemes are well-established. As an immediate corollary, we obtain a broad class of algorithms -- that contains OGDA as a special case -- with a last-iterate convergence guarantee to the space of Nash equilibria of the game. 
\end{abstract}

\clearpage

\pagenumbering{arabic}

\section{Introduction}

The fundamental problem of converging to a Nash equilibrium in multi-agent systems has been a topic of prolific research in several fields, including Mathematics, Economics, Algorithmic Game Theory, Optimization \cite{10.2307/j.ctt1r2gkx,1950PNAS...36...48N,sion1958,10.1145/1132516.1132527,10.5555/1296179,10.1007/s10107-004-0552-5} and more recently Machine Learning in the context of Generative Adversarial Networks (GANs) \cite{NIPS2014_5423,arjovsky2017wasserstein} and multi-agent reinforcement learning \cite{10.5555/645527.657296}. In particular, finding a Nash equilibrium in a two-player \emph{zero-sum} game with objective function $f: \mathcal{X} \times \mathcal{Y} \mapsto \mathbb{R} $ is tantamount to computing a pair of strategies $(\mathbf{x}^*, \mathbf{y}^*) \in \mathcal{X} \times \mathcal{Y}$ such that

\begin{equation}
    f(\mathbf{x}^*, \mathbf{y}) \leq f(\mathbf{x}^*, \mathbf{y}^*) \leq f(\mathbf{x}, \mathbf{y}^*), \quad \forall (\mathbf{x}, \mathbf{y}) \in \mathcal{X} \times \mathcal{Y},
\end{equation}
where $\mathcal{X}$ and $\mathcal{Y}$ represent the space of strategies of each player respectively. In words, an equilibrium -- or a \emph{saddle-point} -- is a pair of strategies such that neither player can improve her utility by a unilateral deviation. In this context, the celebrated min-max theorem implies that if $f(\mathbf{x}, \mathbf{y})$ is a convex-concave function, then an equilibrium always exists; yet, it does not inform us on whether natural learning algorithms can converge to a minimax pair with a reasonable amount of computational resources. 

This question has given rise to profound research, commencing from the analysis of fictitious play by J. Robinson \cite{10.2307/1969530}, and leading to the development of the no-regret paradigm \cite{cesa-bianchi_lugosi_2006,10.1007/11503415_42,v008a006}. However, despite the intrinsic appeal of this framework, these results usually suffer from a lack of last-iterate guarantees, given that a regret-based -- or time-average -- analysis cannot distinguish between a self-stabilizing system and one with recurrent cycles. Indeed, it has been recognized that limit cycles persist in broad classes of no-regret schemes, such as Mirror Descent \cite{mertikopoulos2017cycles,palaiopanos2017multiplicative,10.1145/2840728.2840757, doi:10.1137/1.9781611973402.64}. It is important to point out that a time-average guarantee is insufficient for two main reasons. First, practical applications pose considerable impediments, and implementations of such algorithms could potentially require enormous memory resources; e.g., training GANs through fictitious play would necessitate the storage of millions of parameters at each time step \cite{DBLP:conf/iclr/GrnarovaLLH018,DBLP:conf/eccv/GeXCBW18}. Thus, it is strongly motivated to design algorithms that perform the optimization step through a limited time horizon. Moreover, the stability of the system is the primary desideratum from a control-theoretic perspective, and subsequently of fundamental importance within Game Theory which endeavors to analyze and control the behavior of rational agents.  

These shortcomings of traditional learning algorithms -- such as Gradient Descent/Ascent (GDA) -- were tackled by Daskalakis et al. \cite{DBLP:journals/corr/abs-1711-00141} in the context of training GANs. Specifically, the authors established a pointwise convergence guarantee for Optimistic Gradient Descent/Ascent (henceforth OGDA) -- a natural variant of GDA that augments the memory of the optimization with the previously observed gradient -- for bilinear games, i.e. the objective function is linear in the strategies of both players. This result has subsequently engendered a considerable body of work along several directions \cite{mertikopoulos2018optimistic,liang2018interaction,mokhtari2019unified,daskalakis2018lastiterate,daskalakis2018limit}. Our study follows this long line of research and addresses one such question: Can we extend the last-iterate guarantee of OGDA to a broad class of first-order learning algorithms?  

More precisely, building on the intuition of OGDA, we consider a general family of gradient-based dynamics that augment the memory with multiple historical steps. This consideration is strongly motivated by the no-regret literature, where learning through the previously observed costs lies at the heart of the optimization. Naturally, the main challenge resides primarily in the analysis, and elementary techniques appear to be of no use. In this context, our results are established based on the Z-transform, a powerful tool for deriving a frequency-domain representation of linear dynamics. The main contribution of our paper is to reduce the convergence to a Nash equilibrium -- through the considered class of algorithms -- to the stability of a polynomial, for which efficient algorithmic schemes are well-established in Mathematics and Control Theory. We illustrate the power of our framework in the analysis of OGDA, providing a remarkably simple and precise characterization. 

\paragraph{Related Work}

Our work follows the line of research initiated by Daskalakis et al. \cite{DBLP:journals/corr/abs-1711-00141}; their main contribution was showing through an inductive argument that a natural variant of GDA exhibits last-iterate convergence in \emph{unconstrained} bilinear games. Their approach was inspired by a technique previously introduced and applied in the fields of Online Learning and Convex Optimization \cite{NIPS2013_5147,NIPS2015_5763,pmlr-v32-steinhardtb14,wang2018acceleration}, namely \emph{optimism}. More precisely, this technique is mainly employed when the future cost functions are smooth -- or predictable; e.g., they derive from a regularized algorithm such as Follow-The-Regularized-Leader (FTRL). In this case, it has been shown that optimism can substantially accelerate the dynamics. 

Moreover, Liang and Stokes \cite{liang2018interaction} established that the convergence's rate of OGDA is linear, assuming that the matrix of the game is square and non-singular. This result was also obtained by Mokhtari et al. \cite{mokhtari2019unified}, while they additionally extended the last-iterate guarantee of OGDA in convex-concave minimax problems. In a closely related direction, \cite{mokhtari2019unified} also considered a variant of OGDA with slightly perturbed coefficients. Their approach was founded on the observation that OGDA can be examined as an approximate version of the proximal point method \cite{DBLP:journals/jgo/KaplanT98}.

The more challenging \emph{constrained} case has also been recently addressed in a series of works. First, Mertikopoulos et al. \cite{mertikopoulos2018optimistic} established a convergence guarantee for the projected version of OGDA and Extra-Gradient methods through the Variational Inequality framework. Furthermore, the optimistic variant of multiplicative weights was analyzed by Daskalakis and Panageas \cite{daskalakis2018lastiterate}. A no-regret analysis of Optimistic Mirror Descent can be found in \cite{kangarshahi2018lets}. For a characterization of the stationary points of GDA and OGDA beyond convex-concave settings we refer to \cite{daskalakis2018limit}. Finally, several algorithms have been proposed specifically for solving saddle-point problems; we refer to \cite{adolphs2018local,schfer2019competitive,mazumdar2019finding} and references thereof.

\paragraph{Our Contributions} In the first part of our work (\Cref{section:OGDA}), we analyze Optimistic Gradient Descent/Ascent \cite{DBLP:journals/corr/abs-1711-00141} in bilinear games, strengthening and improving prior results along several lines. Specifically, we derive a concise expression for the characteristic equation of the induced dynamical system (\Cref{proposition:characteristic_equation}); subsequently, we obtain the first -- to the best of our knowledge -- tight analysis of stability with respect to the learning rate, providing the exact region of convergence in \Cref{theorem:convergence}. Our characterization also implies the learning rate that leads to the most rapid convergence (see \Cref{subsection:proof_theorem}), a result which is of clear significance in practical implementations of the algorithm. Our analysis is considerably simpler than the existing ones, and is established through a frequency-domain representation of the dynamics. In particular, we apply the Z-transform, a powerful tool employed in Signal Processing and Control Theory to analyze linear discrete-time systems.  

Moreover, we employ the Z-transform in order to study a general class of gradient-based optimization algorithms in bilinear games (\Cref{section:historical_methods}). More precisely, we first identify sufficient conditions under which the limit points of the dynamics are Nash equilibria (\Cref{proposition:general-nash-limit}). Next, we reduce the stability of the dynamics to the stability of a polynomial(\Cref{theorem:reduction}), for which efficient algorithmic schemes are well-documented in the literature of Mathematics and Control Theory. As an immediate consequence, we show that Optimistic Gradient Descent/Ascent is an instance of a broad class of algorithms that exhibit an analogous convergence guarantee. 

Our results strongly supplement the work of \cite{mokhtari2019unified}, where the authors studied a variant of OGDA with slightly perturbed coefficients. The class of algorithms we consider allows for arbitrary coefficients and importantly, incorporates multiple historical gradient steps -- within a certain bounded time horizon. Naturally, historical methods are of great importance within the literature of Game Theory and the no-regret framework, and providing robust paradigms for their analysis constitutes a crucial endeavor. We also believe that our approach could be of independent interest in the realm of Optimization. Our only hypothesis is that the matrix of the game is square and non-singular, a normative assumption (e.g. see \cite{liang2018interaction,mokhtari2019unified}) made to simplify the analysis.     

\section{Preliminaries}

\paragraph{Optimistic Gradient Descent/Ascent} Consider a continuously differentiable function $f: \mathcal{X} \times \mathcal{Y} \mapsto \mathbb{R}$ that represents the objective function of the game, with $f(\mathbf{x}, \mathbf{y})$ the payoff of player $\mathcal{X}$ to player $\mathcal{Y}$ under strategies $\mathbf{x} \in \mathcal{X}$ and $\mathbf{y} \in \mathcal{Y}$ respectively. Throughout this paper, we focus exclusively on the \emph{unconstrained} case, i.e. $\mathcal{X} = \mathbb{R}^n$ and $\mathcal{Y} = \mathbb{R}^m$. Perhaps the most natural optimization algorithm for solving the induced saddle-point problem is by performing simultaneously Gradient Descent on $\mathbf{x}$ and Gradient Ascent on $\mathbf{y}$; formally, if $\eta>0$ denotes some positive constant -- typically referred to as the \emph{learning rate}, GDA can be described as follows:
\begin{equation}
    \label{equation:GDA}
    \begin{split}
    \mathbf{x}_{t} = \mathbf{x}_{t-1} - \eta \nabla_{\mathbf{x}} f(\mathbf{x}_{t-1}, \mathbf{y}_{t-1}), \\
    \mathbf{y}_{t} = \mathbf{y}_{t-1} + \eta \nabla_{\mathbf{y}} f(\mathbf{x}_{t-1}, \mathbf{y}_{t-1}).
    \end{split}
\end{equation}

However, there are very simple examples where the system of equations \eqref{equation:GDA} diverges; for instance, when $f(x, y) = x y$ with $x, y \in \mathbb{R}$ and $(x_0, y_0) \neq (0, 0)$, GDA is known to diverge for any learning rate $\eta > 0$. This inadequacy has motivated \emph{optimistic} variants of GDA that incorporate some prediction on the next iteration's gradient through the regularization term (recall that Gradient Descent can be viewed as an instance of FTRL with Euclidean regularizer \cite{MAL-018}). With OGDA we refer to the optimization variant that arises when the prediction of the next iteration's gradient is simply the previously observed gradient; this yields the following update rules:
\begin{equation}
    \label{equation:OGDA}
    \begin{split}
    \mathbf{x}_{t} = \mathbf{x}_{t-1} - 2 \eta \nabla_{\mathbf{x}} f(\mathbf{x}_{t-1}, \mathbf{y}_{t-1}) + \eta \nabla_{\mathbf{x}} f(\mathbf{x}_{t-2}, \mathbf{y}_{t-2}), \\
    \mathbf{y}_{t} = \mathbf{y}_{t-1} + 2 \eta \nabla_{\mathbf{y}} f(\mathbf{x}_{t-1}, \mathbf{y}_{t-1}) - \eta \nabla_{\mathbf{y}} f(\mathbf{x}_{t-2}, \mathbf{y}_{t-2}).
    \end{split}
\end{equation}

We remark that in our analysis of OGDA (\Cref{section:OGDA}) we allow $\eta$ to take negative values, although this consideration might appear counter-intuitive. In a \emph{bilinear} game the objective function is linear in the strategies of both players, i.e. $f(\mathbf{x}, \mathbf{y}) = \mathbf{x}^T A \mathbf{y}$; here, $A$ denotes an $n \times m$ real matrix which will be referred to as the matrix of the game. We say that the dynamics are \emph{stable} if the strategies of both players converge.

\subsection{The Z-transform}
\label{section:Z}

The purpose of this subsection is to review the Z-transform, and certain of its properties that will be employed in the forthcoming parts of our study. Specifically, this tool is widely applied in Mathematics and Signal Processing in order to convert a discrete-time signal into a frequency-domain representation; one should view it as the discrete-time analog of the Laplace transform \cite{murray1942}. Importantly for this wok, the Z-transform can greatly simplify the analysis and the characterization of discrete-time systems. Indeed, in Control Theory the designer predominantly regulates the system's behavior in the frequency domain; for an introduction to the basic principles of Control Theory we refer to \cite{1453566,Yoon2013}. We should also point out that the Z-transform has numerous applications in Combinatorics, where the \emph{unilateral} expression is known as \emph{ordinary generating function} (OGF) of a sequence; e.g. see \cite{10.5555/1506267,10.5555/1984890}.    

Consider a discrete-time signal $x_n \in \mathbb{R}$; we will assume that $x_n$ is \emph{causal}, that is $x_n = 0$ for every $n < 0$. For causal signals, the Z-transform coincides with the unilateral -- or \emph{one-sided} -- Z-transform, and is defined as 

\begin{equation}
    \label{equation:Z-transform}
    X(z) = \mathcal{Z} \{ x_n \} =  \sum_{n=- \infty}^{+ \infty} x_n z^{-n} = \sum_{n=0}^{+ \infty} x_n z^{-n},
\end{equation}
where we assume that the frequency parameter $z$ resides in the region of convergence:

\begin{equation}
    \text{ROC} = \left\{ z \in \mathbb{C}^* : \left| \sum_{n=0}^{+\infty} x_n z^{-n} \right| < + \infty \right\}.
\end{equation}

In our setting, the dynamics involve multi-dimensional signals that represent the players' strategies throughout the evolution of the game. In this case, the Z-transform is defined by applying the definition \eqref{equation:Z-transform} in each coordinate or dimension of the signal. Moreover, a crucial feature of the -- unilateral -- Z-transform is the following time delay property:

\begin{equation}
    \label{equation:time_delay}
    \mathcal{Z} \{ x_{n+k} \} = z^k \left( X(z) - \sum_{i=0}^{k-1} x_i z^{-i} \right),
\end{equation}
for some $k \in \mathbb{N}$ and $x_0, x_1, \dots, x_{k-1}$ the initial conditions. This property greatly simplifies the analysis of linear discrete-time systems, as it will be illustrated in the forthcoming sections. We also remark the following special cases of \eqref{equation:time_delay} that will be used in the analysis of OGDA:

\begin{itemize}
    \item $\mathcal{Z} \{ x_{n+1} \} = z (X(z) - x_0)$
    \item $\mathcal{Z} \{ x_{n+2} \} = z^2 (X(z) - x_0 - x_1 z^{-1}) $
\end{itemize}

Of course, it is easy to see that these properties also extend for the multi-dimensional definition. In this context, we state a fundamental theorem in Signal Processing that connects the asymptotic behavior of a causal signal with its Z-transform.

\begin{theorem}
    \label{theorem:fundamental_theorem}
Let $x_n$ a causal discrete-time signal with $X(z) = \mathcal{Z}\{ x_n \}$ a rational function; then, $x_n$ converges for $n \to \infty$ if and only if every pole of $X(z)$ resides inside the unit circle\footnote{Recall that the unit circle in the complex plane is defined as $\mathbb{U}_1 = \{ z \in \mathbb{C} : |z| = 1 \}$.}, except from at most one pole at $z=1$.
\end{theorem}

Note that the poles of a rational function in reduced form -- the denominator has no common factors with the nominator -- coincide with the roots of the polynomial in the denominator. As a result, the convergence of the dynamics can be reduced to analyzing the roots of a polynomial. For a comprehensive study on the Z-transform and applications in discrete-time systems we refer to \cite{10.5555/201033,6371274,trove.nla.gov.au/work/34670171} and references thereof.

\paragraph{Transfer Function} Consider a linear discrete-time system that maps an input signal $x_t$ to an output signal $y_t$ through the following difference equation:

\begin{equation}
    a_n y_{t+n} + \dots + a_0 y_{t} = b_m x_{t+m} + \dots + b_0 x_t,
\end{equation}
for some real coefficients $a_0, \dots, a_n, b_0, \dots, b_m$. The transfer function of the system $H(z)$ is defined as 

\begin{equation}
    H(z) = \frac{Y(z)}{X(z)} = \frac{b_0 + b_1 z + \dots + b_m z^m}{a_0 + a_1 z + \dots + a_n z^n}.
\end{equation}
Notice that in the definition of the transfer function, the initial conditions are taken to be zero. 

\paragraph{Notation}

We use lowercase boldface for time domain vectors and capital boldface for frequency domain vectors. We represent with $t$ the discrete time, while subscripts mainly refer to time indices. We use $\mathbf{I}_{k}$ and $\mathbf{0}_{k \times \ell}$ to refer to the identity matrix of size $k \times k$ and the zero matrix of size $k \times \ell$ respectively; when $k = \ell$ we simply write $\mathbf{0}_{k}$ instead of $\mathbf{0}_{k \times k}$. For a matrix $S$, we denote with $|| S ||$ the \emph{spectral norm} of $S$, i.e. the square root of the maximum eigenvalue of $S^T S$. Finally, a polynomial is said to be stable if all of its roots reside within the unit circle in the complex plane.


\section{Analysis of OGDA}
\label{section:OGDA}

In this section, we employ the Z-transform in order to analyze OGDA in bilinear and unconstrained zero-sum games; that is, $f(\mathbf{x}, \mathbf{y}) = \mathbf{x}^T A \mathbf{y}$, $\mathcal{X} = \mathbb{R}^n$, and $\mathcal{Y} = \mathbb{R}^m$. We will allow arbitrary initial conditions, represented with $\mathbf{x}_0, \mathbf{x}_1 \in \mathbb{R}^n$ and $\mathbf{y}_0, \mathbf{y}_1 \in \mathbb{R}^m$ respectively. In this setting, the update rules of OGDA \eqref{equation:OGDA} can be described for any $t \geq 0$ as follows:

\begin{equation}
    \label{equation:OGDA-bilinear}
    \begin{split}
    \mathbf{x}_{t+2} &= \mathbf{x}_{t+1} - 2 \eta A \mathbf{y}_{t+1} + \eta A \mathbf{y}_{t}, \\
    \mathbf{y}_{t+2} &= \mathbf{y}_{t+1} + 2 \eta A^T \mathbf{x}_{t+1} - \eta A^T \mathbf{x}_{t}.
    \end{split}
\end{equation}
These equations can be expressed more concisely in matrix form as 

\begin{equation}
    \label{equation:OGDA-matrix_form}
    \begin{pmatrix}
    \mathbf{x}_{t+2} \\
    \mathbf{y}_{t+2} 
    \end{pmatrix}
    =
    \begin{pmatrix}
    \mathbf{I}_n & -2 \eta A \\
    2 \eta A^T & \mathbf{I}_m
    \end{pmatrix}
    \begin{pmatrix}
    \mathbf{x}_{t+1} \\
    \mathbf{y}_{t+1} 
    \end{pmatrix}
    + 
    \begin{pmatrix}
    \mathbf{0}_n & \eta A \\
    -\eta A^T & \mathbf{0}_m
    \end{pmatrix}
    \begin{pmatrix}
    \mathbf{x}_{t} \\
    \mathbf{y}_{t} 
    \end{pmatrix}.
\end{equation}

\bigbreak
\noindent
In correspondence to the last expression, let us introduce the following matrices:

\begin{equation}
    B =
    \begin{pmatrix}
    \mathbf{I}_n & -2 \eta A \\
    2\eta A^T & \mathbf{I}_m
    \end{pmatrix},
    C = 
    \begin{pmatrix}
    \mathbf{0}_{n} & \eta A \\
    -\eta A^T & \mathbf{0}_m
    \end{pmatrix},
    \mathbf{w}_t = 
    \begin{pmatrix}
    \mathbf{x}_t \\
    \mathbf{y}_t
    \end{pmatrix}.
\end{equation}

\bigbreak
\noindent
With this notation, \Cref{equation:OGDA-matrix_form} can be re-written as

\begin{equation}
    \label{equation:recursion-second_order}
    \mathbf{w}_{t+2} = B \mathbf{w}_{t+1} + C \mathbf{w}_{t}.
\end{equation}

Let us denote with $\mathbf{W}(z) = \mathcal{Z} \{ \mathbf{w}_t \}$ the Z-transform of $\mathbf{w}_t$; if we transfer \Cref{equation:recursion-second_order} in the frequency domain and apply the time delay property \eqref{equation:time_delay} it follows that 

\begin{align*}
    z^2 ( \mathbf{W}(z) - \mathbf{w}_0 - \mathbf{w}_1 z^{-1}) &= z B (\mathbf{W}(z) - \mathbf{w}_0) + C \mathbf{W}(z) \iff \\
    (z^2 \mathbf{I}_{n+m} - z B - C) \mathbf{W}(z) &= (z^2 \mathbf{I}_{n+m} - z B) \mathbf{w}_0  + z \mathbf{w}_1,
\end{align*}
where we assume that $z$ resides in the region of convergence of $\mathbf{W}(z)$. Therefore, it is easy to see that $\mathbf{W}(z)$ can be expressed as a rational function in every dimension; indeed, recall that for a non-singular and square matrix $S$

\begin{equation}
    \label{equation:inverse}
    S^{-1} = \frac{\adj(S)}{\det(S)},
\end{equation}
where if $M_{ij}$ denotes the $(i,j)$-minor of $S$ - the determinant of the matrix that results from removing the $i^{\text{th}}$ row and the $j^{\text{th}}$ column from $S$, $\text{adj}(S)_{ij} = ( (-1)^{i+j} M_{ij} )^T$. As a result, each entry of matrix $\adj(z^2 \mathbf{I}_{n+m} - z B - C)$ is a polynomial of $z$ and hence, $\mathbf{W}(z)$ can be expressed - in each dimension - as the ratio of two polynomials. Note that this property holds for any sequence that satisfies a linear difference equation. Thus, the characteristic equation of the system can be expressed as 

\begin{equation}
    \label{equation:determinant_1}
    \det(z^2 \mathbf{I}_{n+m} - z B - C) =
    \begin{vmatrix}
    (z^2 - z) \mathbf{I}_n & (2z-1) \eta A \\
    - (2z - 1) \eta A^T & (z^2 - z) \mathbf{I}_m
    \end{vmatrix}
    = 0.
\end{equation}
With the assumption that $m = n$, this determinant can be expanded as 

\begin{equation}
    \label{equation:determinant_2}
    \begin{vmatrix}
    (z^2 - z) \mathbf{I}_n & (2z-1) \eta A \\
    - (2z - 1) \eta A^T & (z^2 - z) \mathbf{I}_m
    \end{vmatrix} = 
    \det \left( (z^2 - z)^2 \mathbf{I}_n + (2z-1)^2 \eta^2 A A^T \right),
\end{equation}
where we employed a well-known identity for the determinant of a block matrix. Let $\alpha(z)$ represent the characteristic polynomial of matrix $\eta^2 A A^T$ -- or equivalently $\eta^2 A^T A$; given that $z=1/2$ is not a root of \Cref{equation:determinant_1} it follows that

\begin{equation}
    \eqref{equation:determinant_1} \iff \det\left( \left( \frac{z^2 - z}{2z - 1} \right)^2 \mathbf{I}_n + \eta^2 A A^T \right) = 0 \iff \alpha \left( - \left( \frac{z^2 - z}{2z -1} \right)^2 \right) = 0.
\end{equation}

As a result, we have derived a concise expression for the characteristic equation of the system as a function of the spectrum of matrix $A$ and the learning rate $\eta$. 

\begin{proposition}
    \label{proposition:characteristic_equation}
Consider a square matrix $A$, and some learning rate $\eta$; if $\alpha(z)$ represents the characteristic polynomial of matrix $\eta^2 A A^T$, the characteristic equation of OGDA in bilinear games \eqref{equation:OGDA-bilinear} can be expressed as 

\begin{equation}
    \label{equation:spectrum_transformation}
    \chi(z) = \alpha \left( - \left( \frac{z^2 - z}{2z - 1} \right)^2 \right) = 0.
\end{equation}
\end{proposition}

This expression provides a complete characterization of the stability of the system, as well as the rate of convergence. To be more precise, let $\mathcal{P}$ the set of the roots of $\chi(z) = 0$. For a non-singular matrix $A$ it follows that $1 \notin \mathcal{P}$; hence, we know from \Cref{theorem:fundamental_theorem} that the dynamics converge if and only if $|z| < 1, \forall z \in \mathcal{P}$. Moreover, assuming that the system is stable, the rate of convergence is determined by the root with the largest norm; i.e. if $\lambda$ is the rate of convergence, it follows that $\lambda = \max \{ |z| : z \in \mathcal{P} \}$. We should also remark that the region of convergence of $\mathbf{W}(z)$ can be expressed as

\begin{equation}
    \text{ROC} = \{ z \in \mathbb{C}^{*} : |z| > \max \{ |z'| : z' \in \mathcal{P} \} \}.
\end{equation}

Through this prism, the stability criterion for the system is equivalent to $e^{j \theta} \in \text{ROC}$ -- for any $\theta$. Having established \Cref{proposition:characteristic_equation}, the remainder of the analysis reduces to trivial algebraic calculations (see \Cref{subsection:proof_theorem}). In particular, considering the non-trivial case of $\gamma = ||A|| \neq 0$, we prove the following theorem:

\begin{theorem}
\label{theorem:convergence}
Consider a square and non-singular matrix $A$. For any learning rate $\eta$ such that $|\eta| < 1/(\sqrt{3} \gamma)$ OGDA \eqref{equation:OGDA-bilinear} converges from any initial state, whilst for $\eta$ such that $|\eta| \geq 1/(\sqrt{3} \gamma)$ the behavior of the dynamics is inherently unstable \footnote{The instability is implied in the sense that at least under certain initial states the dynamics will not converge.}.
\end{theorem}

A surprising consequence of this theorem is that the stability of the system and indeed, the rate of convergence depends on $\eta^2$; thus, the dynamics could converge with negative learning rate (see \Cref{fig:learning_rate}), a very counter-intuitive property in the context of Gradient Descent. 

We provide a simple example to illustrate the main implications of \Cref{theorem:convergence}. In particular, we consider the objective function $f(x,y) = x y$ for $x, y \in \mathbb{R}$. We exhibit the qualitative behavior of OGDA and the impact of the learning rate in Figures \ref{fig:learning_rate} and \ref{fig:stability}. 

\begin{figure}[!ht]
    \centering
    \includegraphics[scale=0.45]{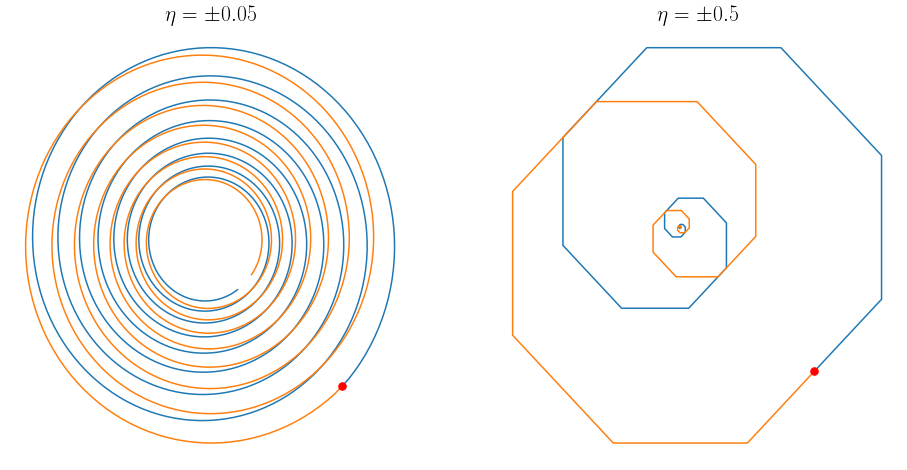}
    \caption{The behavior of the dynamics for different values of learning rate. Note that the blue spiral corresponds to a positive $\eta$, while the orange spiral to the corresponding negative. The red point represents the initial state of the system. It is important to point out that \Cref{proposition:characteristic_equation} implicitly provides the learning rate that leads to the most rapid convergence, establishing an exact characterization of the convergence's rate with respect to $\eta$. In our simple example, the optimal learning rate is $\eta = 0.5$, illustrated in the rightmost image.}
    \label{fig:learning_rate}
\end{figure}

\begin{figure}[!ht]
    \centering
    \includegraphics[scale=0.45]{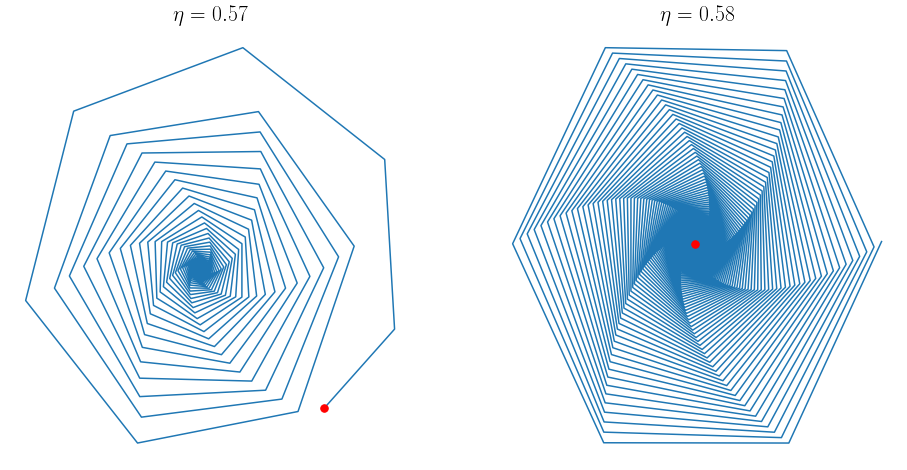}
    \caption{The behavior of the dynamics near the critical value of the learning rate. In particular, according to \Cref{theorem:convergence} the system converges only if $\eta < 1/\sqrt{3} \simeq 0.577$. Thus, we can indeed verify the transition of the system from stability (leftmost image) to divergent behavior (rightmost image). We also remark that the convergence of the system becomes gradually slower as the learning rate exceeds the optimal value of $\eta = 0.5$ and advances towards the critical point.}
    \label{fig:stability}
\end{figure}

Finally, we conclude this section with a characterization of the limit points of OGDA; specifically, it is easy to establish the following claim.

\begin{proposition}
    \label{proposition:nash-limit}
If OGDA converges, the limit points are Nash equilibria \footnote{This proposition posits that $\eta \neq 0$; otherwise, the limit points coincide with the initial state}.
\end{proposition}

\begin{proof}
We first remark that a pair of strategies $(\mathbf{x}^*, \mathbf{y}^*)$ is a Nash equilibrium of the bilinear game if and only if $A \mathbf{y}^* = \mathbf{0}$ and $A^T \mathbf{x}^* = \mathbf{0}$. Assuming that the dynamics converge, we denote with $\mathbf{x}_{\infty} = \lim_{t \to + \infty} \mathbf{x}_t$ and $\mathbf{y}_{\infty} = \lim_{t \to + \infty} \mathbf{y}_t$. It follows from \Cref{equation:OGDA-bilinear} that

\begin{equation}
    \lim_{t \to + \infty} (\mathbf{x}_{t+1} - \mathbf{x}_t) = \mathbf{0} = - 2\eta A \mathbf{y}_{\infty} + \eta A \mathbf{y}_{\infty} \implies A \mathbf{y}_{\infty} = \mathbf{0},
\end{equation}

\begin{equation}
    \lim_{t \to + \infty} (\mathbf{y}_{t+1} - \mathbf{y}_t) = \mathbf{0} = - 2\eta A^T \mathbf{x}_{\infty} + \eta A^T \mathbf{x}_{\infty} \implies A^T \mathbf{x}_{\infty} = \mathbf{0}.
\end{equation}

\end{proof}

We should point that if $A$ is square and non-singular the unique Nash equilibrium of the game is at $(\mathbf{x}^*, \mathbf{y}^*) = (\mathbf{0}, \mathbf{0})$; nonetheless, \Cref{proposition:nash-limit} holds for any arbitrary matrix $A$. In the following section we provide a general characterization of the limit points for a broad class of linear optimization algorithms, extending the result of \Cref{proposition:general-nash-limit}.

\section{Historical Methods}
\label{section:historical_methods}

In this section, we analyze a general class of linear gradient-based optimization algorithms. More precisely, consider some constant parameter $k$ that corresponds to the time horizon of each player. Note that for Gradient Descent/Ascent this parameter would be $k=1$ as the step of each player depends only on the current state of the system, while the optimistic variant we analyzed in the previous section has a $k=2$ time horizon since it incorporates an additional historical gradient. More broadly, we consider the following update rules for $t \geq 0$:

\begin{equation}
    \label{equation:HGDA}
    \begin{split}
    \mathbf{x}_{t+k} = \sum_{i=1}^k p_i \mathbf{x}_{t+k-i} - \eta \sum_{i=1}^k q_i \nabla_{\mathbf{x}} f(\mathbf{x}_{t+k-i}, \mathbf{y}_{t+k-i}), \\
    \mathbf{y}_{t+k} = \sum_{i=1}^k p_i \mathbf{y}_{t+k-i} + \eta \sum_{i=1}^k q_i \nabla_{\mathbf{y}} f(\mathbf{x}_{t+k-i}, \mathbf{y}_{t+k-i}),
    \end{split}
\end{equation}
where $\mathbf{x}_0, \dots, \mathbf{x}_{k-1} \in \mathbb{R}^n$ and $\mathbf{y}_0, \dots, \mathbf{y}_{k-1} \in \mathbb{R}^m$ the initial conditions of the system, and $p_1, \dots, p_k, q_1, \dots q_k \in \mathbb{R}$ some real parameters. These dynamics will be referred to as Historical Gradient Descent/Ascent (henceforth HGDA). Let us elaborate on \Cref{equation:HGDA} for the case of bilinear games. Specifically, one should distinguish the following two main components in the dynamics:

\paragraph{Smoothness Term} The first component of the HGDA dynamics is the smoothness or exploitation term, and determines how the previously followed strategies within the time horizon regularize the forthcoming action. For instance, in OGDA (or indeed GDA) this particular term for the $\mathbf{}$x player would take the form $\mathcal{S} \{ \mathbf{x}_t \} = \mathbf{x}_{t+1} - \mathbf{x}_t$. In general, the smoothness term can be expressed as 

\begin{equation}
    \label{equation:smoothness}
    \mathcal{S} \{ \mathbf{x}_t \} = \mathbf{x}_{t+k} -  p_1 \mathbf{x}_{t+k-1} - \dots - p_k \mathbf{x}_{t}.
\end{equation}
The smoothness term can be captured through the transfer function $S(z)$ of system \eqref{equation:smoothness}; in particular, it follows that 

\begin{equation}
    S(z) = z^k - p_1 z^{k-1} - p_2 z^{k-2} - \dots - p_k.
\end{equation}

\paragraph{Gradient Term} The second component of the HGDA dynamics is the gradient or exploration term, and determines how each player explores the space of strategies through the previously observed gradients. Notice that for bilinear games $\nabla_{\mathbf{x}} f(\mathbf{x}, \mathbf{y}) = A \mathbf{y}$ and $\nabla_{\mathbf{y}} f(\mathbf{x}, \mathbf{y}) = A^T \mathbf{x}$; thus, the gradient term follows after applying a linear transformation to the opponent's previous strategies. Let us focus on the weighted average of the opponent's strategies within the time horizon. For OGDA, this weighted average would be $\mathcal{G} \{ \mathbf{y}_t \} = 2 \mathbf{y}_{t+1} - \mathbf{y}_{t}$; in general, it can be expressed as  

\begin{equation}
    \label{equation:gradient}
    \mathcal{G}\{ \mathbf{y}_t \} = q_1 \mathbf{y}_{t+k-1} + \dots + q_k \mathbf{y}_t.
\end{equation}
\bigbreak
\noindent
Similarly, the transfer function of system \eqref{equation:gradient} can be written as 

\begin{equation}
    G(z) = q_1 z^{k-1} + q_2 z^{k-2} + \dots + q_k. 
\end{equation}
As a result, HGDA in bilinear games obtains the following form:

\begin{equation}
    \label{equation:HGDA-bilinear}
    \begin{split}
    \mathcal{S}\{ \mathbf{x}_t \} = - \eta A \mathcal{G} \{ \mathbf{y}_t \}, \\
    \mathcal{S} \{ \mathbf{y}_t \} =  \eta A^T \mathcal{G} \{ \mathbf{x}_t \}.
    \end{split}
\end{equation}

First, we provide sufficient conditions under which the study of the induced dynamical system \eqref{equation:HGDA-bilinear} is meaningful from a game-theoretic standpoint. More precisely, we prove the following proposition: 

\begin{proposition}
\label{proposition:general-nash-limit}
Consider the HGDA dynamics \eqref{equation:HGDA-bilinear} such that $S(1) = 0$ and $G(1) \neq 0$; then, if the dynamics converge, the limit points are Nash equilibria \footnote{We again assume that $\eta \neq 0$.}.
\end{proposition}

\begin{proof}
First, we have that 

\begin{equation}
    S(1) = 0 \iff \sum_{i = 1}^k p_i = 1,
\end{equation}

\begin{equation}
    G(1) \neq 0 \iff \sum_{i=1}^k q_i = q_0 \neq 0.
\end{equation}

Recall that a pair of strategies $(\mathbf{x}^*, \mathbf{y}^*)$ is a Nash equilibrium if and only if $A \mathbf{y}^* = \mathbf{0}$ and $A^T \mathbf{x}^* = \mathbf{0}$. Let $\mathbf{x}_{\infty} = \lim_{t \to + \infty} \mathbf{x}_t$ and $\mathbf{y}_{\infty} = \lim_{t \to + \infty} \mathbf{y}_t$; it follows that 

\begin{equation}
    \lim_{t \to + \infty} \mathcal{S} \{ \mathbf{x}_t \} = \mathbf{x}_{\infty} \left( 1 - \sum_{i=1}^k p_i \right) = \mathbf{0},
\end{equation}

\begin{equation}
    \lim_{t \to + \infty} \mathcal{G} \{ \mathbf{y}_t \} = \mathbf{y}_{\infty} \sum_{i=1}^k q_i = q_0 \mathbf{y}_{\infty}.
\end{equation}

Therefore, it follows from \eqref{equation:HGDA-bilinear} that $\mathbf{0} = - \eta q_0 A \mathbf{y}_{\infty} \iff A \mathbf{y}_{\infty} = \mathbf{0}$, since we have assumed that $\eta \neq 0$. Analogously, we can show that $A^T \mathbf{x}_{\infty} = \mathbf{0}$.
\end{proof}

As a result, we have reduced finding the Nash equilibrium to the stability of the dynamical system \eqref{equation:HGDA-bilinear}. It is also easy to establish the following proposition:

\begin{proposition}
    \label{proposition:Z-space_formulation}
    The HGDA dynamics \eqref{equation:HGDA-bilinear} can be transformed in the z-space as
    
    \begin{equation}
        \begin{pmatrix}
        S(z) \mathbf{I}_n & G(z) \eta A \\
        - G(z) \eta A^T & S(z) \mathbf{I}_m
        \end{pmatrix}
        \mathbf{W}(z)
        = 
        \mathcal{Q}(z; \mathbf{w}_0, \mathbf{w}_1, \dots, \mathbf{w}_{k-1}),
    \end{equation}
    where $\mathcal{Q}(z) \in \mathbb{R}^{n+m}$ is a polynomial of $z$ in each dimension that depends on the initial state of the system.
\end{proposition}

This proposition follows directly after transferring \eqref{equation:HGDA-bilinear} to the z-space and applying the generalized time delay property \eqref{equation:time_delay}. Moreover, \Cref{proposition:Z-space_formulation} implies that the characteristic equation of HGDA in bilinear games can be expressed as 

\begin{equation}
    \label{equation:characteristic-HGDA}
    \begin{vmatrix}
    S(z) \mathbf{I}_n & G(z) \eta A \\
    - G(z) \eta A^T & S(z) \mathbf{I}_m
    \end{vmatrix}
    = 0.
\end{equation}

It is clear that every common root of $G(z)$ and $S(z)$ will also satisfy \Cref{equation:characteristic-HGDA}, yielding an additional constraint on the structure of the optimization algorithm. In particular, let $\mathcal{P}$ denote the set of all common roots of $G(z)$ and $S(z)$, and consider the following polynomial:

\begin{equation}
    P(z) = \prod_{\rho \in \mathcal{P}} (z - \rho).
\end{equation}

In this context, we proceed with the following factorization: $G(z) = P(z) G'(z)$ and $S(z) = P(z) S'(z)$; notice that the polynomials $G'$ and $S'$ do not have common roots. Applying this factorization yields that 

\begin{equation}
    \label{equation:interm}
    \eqref{equation:characteristic-HGDA} \iff \left( P(z) \right)^{n+m}
    \begin{vmatrix}
    S'(z) \mathbf{I}_n & G'(z) \eta A \\
    - G'(z) \eta A^T & S'(z) \mathbf{I}_m
    \end{vmatrix}
    = 0.
\end{equation}

As a result, we are able to reduce the stability of Historical Gradient Descent/Ascent in bilinear games to the stability of a polynomial, as stated in the following theorem:

\begin{theorem}
    \label{theorem:reduction}
Consider a square and non-singular matrix $A$, and let $x^n + a_1 x^{n-1} + \dots + a_n$ the characteristic polynomial of matrix $\eta^2 A A^T$; then, converging to a Nash equilibrium through HGDA in bilinear games reduces to the stability of the following polynomial:

\begin{equation}
    \label{equation:poly-reduction}
    (-S^2(z))^{n} + a_1 (-S^2(z))^{n-1} (G^2(z)) + \dots + a_n (G^2(z))^{n},
\end{equation}
where $S(z)$ and $G(z)$ are the transfer functions of the smoothness and the gradient term respectively with $S(1) = 0$ and $G(1) \neq 0$.
\end{theorem}

\begin{proof}
First, we have that 

\begin{equation}
    \label{equation:interm_p}
    \begin{vmatrix}
    S'(z) \mathbf{I}_n & G'(z) \eta A \\
    - G'(z) \eta A^T & S'(z) \mathbf{I}_m
    \end{vmatrix}
    = 
    \det\left( (S'(z))^2 \mathbf{I}_n + (G'(z))^2 \eta^2 A A^T \right).
\end{equation}

Let $\mathcal{R}$ denote the set of roots of $G'(z) = 0$. If $\alpha(z)$ represents the characteristic polynomial of matrix $\eta^2 A A^T$, it follows that for $z \notin \mathcal{R}$

\begin{equation}
    \label{equation:interm_pp}
    \eqref{equation:interm} \iff
    (P(z))^{2n} \alpha \left( - \left( \frac{S'(z)}{G'(z)} \right)^2 \right) = 0.
\end{equation}
Moreover, for $z \notin \mathcal{R}$ we obtain that 

\begin{equation}
    \label{equation:interm_ppp}
    \eqref{equation:interm_pp} \iff
    (-S^2(z))^{n} + a_1 (-S^2(z))^{n-1} (G^2(z)) + \dots + a_n (G^2(z))^{n} = 0.
\end{equation}

Finally, consider some $\rho \in \mathcal{R}$. It is clear that $\rho$ satisfies \eqref{equation:interm} if and only if $\rho \in \mathcal{P}$; as a result, \Cref{equation:interm} is equivalent to \Cref{equation:interm_ppp} without the constraint $z \notin \mathcal{R}$, and the theorem follows directly from \Cref{proposition:general-nash-limit}.
\end{proof}

\paragraph{Singular Case} Our reduction (\Cref{theorem:reduction}) is only applicable for a square and non-singular matrix. Indeed, let us assume that $\det(A) = 0$. It follows that $z=1$ is a (multiple) root of \Cref{equation:poly-reduction} and hence, the induced polynomial is not stable. However, recall from \Cref{theorem:fundamental_theorem} that the dynamics could converge with a factor of $(z-1)$ in the denominator; in fact, this term determines the steady state solution -- or the limit points -- of the dynamics. Yet, handling the singular case appears to be technically challenging in the z-space, as one would have to argue about the multiplicity of the factor $(z-1)$ in the adjoint matrix $\adj(S^2(z) \mathbf{I}_n + G^2(z)\eta^2 A A^T)$ (recall \eqref{equation:inverse}).

\paragraph{Stability of a Polynomial} The problem of characterizing the stability of a polynomial has been thoroughly investigated and understood in Mathematics and Control Theory. First, although deriving closed-form solutions for arbitrary polynomials is impossible (Abel's Theorem), numerical methods yield very accurate approximations \cite{10.1145/321062.321064}; in this context, \Cref{theorem:reduction} also implies the rate of convergence -- up to some small error. Moreover, the literature of Control Theory provides efficient algorithmic schemes for characterizing the stability of an arbitrary polynomial. Most notably for discrete-time systems, we mention the Bistritz criterion and the Jury test. The details of these methods extend beyond the scope of our work; we refer the interested reader to  \cite{1084319,1457261,PREMARATNE1993165,989164} and references thereof. 

\bigbreak
Finally, as a concrete example we consider the set of HGDA algorithms with $S(z) = z(z-1)P(z)$ and $G(z) = (2z-1)P(z)$, where $P(z)$ is a stable polynomial; notice that $S(1) = 0$ and $G(1) \neq 0$, as required by \Cref{proposition:general-nash-limit}. It should be clear that OGDA is an instance of this class for $P(z) = 1$. In this case, the convergence of the dynamics is implied -- for any $\eta < 1/(\sqrt{3}\gamma)$ -- by \Cref{theorem:convergence} and \Cref{theorem:reduction}. As a result, we have identified an infinite class of optimization algorithms with a convergence guarantee to the Nash equilibrium of the game.

\section{Concluding Remarks}
\label{section:conclusion}

The main contribution of our paper is twofold. First, we provided an exact characterization of Optimistic Gradient Descent/Ascent in bilinear games through a remarkably simple analysis. More importantly, we introduced a novel class of optimization algorithms -- namely Historical Gradient Descent/Ascent, and we established an algorithmic procedure for evaluating the stability of the dynamics in bilinear games. Our results have been founded on the Z-transform, a robust tool for characterizing linear discrete-time systems. 

We believe that the stability of OGDA -- and variants thereof -- can be better understood through the lens of Control Theory. In particular, the optimistic term should be construed as a regulator that stabilizes the system. An interesting avenue for future research would be to stabilize Gradient Descent/Ascent via explicit techniques from Control Theory; e.g. the celebrated method of PID control. Moreover, our work has left open the case of a singular or non-square matrix, although this direction may be technically challenging. Finally, our frequency-domain framework can be directly applied to analyze and control continuous-time dynamics through the Laplace transform. 

\bibliography{refs.bib}

\appendix

\section{Omitted Proofs}
\label{section:proofs}







\subsection{Proof of \texorpdfstring{\Cref{theorem:convergence}}{}}
\label{subsection:proof_theorem}

\begin{proof}
Let $\Lambda$ be the set of eigenvalues of matrix $\eta^2 A A^T$ -- or equivalently $\eta^2 A^T A$ -- and let $\lambda \in \Lambda$; notice that $\lambda > 0$, given that $\eta^2 A A^T$ is symmetric and positive definite. \Cref{proposition:characteristic_equation} implies that $\lambda$ induces a root $z$ in the characteristic equation of the system \eqref{equation:OGDA-bilinear} such that 

\begin{equation}
    \left( \frac{z^2 - z}{2z - 1} \right)^2 = - \lambda \iff \frac{z^2 - z}{2z - 1} = \pm \sqrt{\lambda} j,
\end{equation}
where $j$ denotes the imaginary unit. This equation yields the following quadratic equations:

\begin{equation}
    \label{equation:quadratic_1}
    z^2 - z (1 + 2\sqrt{\lambda} j) + \sqrt{\lambda}j = 0,
\end{equation}

\begin{equation}
    \label{equation:quadratic_2}
    z^2 - z (1 - 2\sqrt{\lambda} j) - \sqrt{\lambda}j = 0.
\end{equation}

It is easy to see that a complex number $z$ satisfies \Cref{equation:quadratic_1} if and only if its conjugate $\Bar{z}$ is a solution to \Cref{equation:quadratic_2}; hence, it suffices to restrict our attention to the former quadratic equation, given that the induced solutions have the same norm. In particular, \Cref{equation:quadratic_1} has the following solutions:

\begin{equation}
    z_{1,2}(\lambda) = \frac{1 + 2\sqrt{\lambda} j \pm \sqrt{1 - 4\lambda}}{2}.
\end{equation}
We first consider the case where $\lambda \in \left(0, \frac{1}{4}\right]$; specifically, it follows that 

\begin{equation}
    \label{equation:z_1-1}
    |z_1(\lambda)| = \frac{\sqrt{2 + 2\sqrt{1 - 4\lambda}}}{2},
\end{equation}

\begin{equation}
    |z_2(\lambda)| = \frac{\sqrt{2 - 2\sqrt{1 - 4\lambda}}}{2}.
\end{equation}

As a result, for all $\lambda \in \left(0, \frac{1}{4}\right]$ it holds that $|z_1(\lambda)|, |z_2(\lambda)| < 1$ (see \Cref{fig:complex_norm}). Next, we investigate the case where $\lambda > 1/4$; in particular, we obtain that  

\begin{equation}
    \label{equation:z_1-2}
    |z_1(\lambda)| = \sqrt{2\lambda + \sqrt{\lambda(4\lambda-1)}},
\end{equation}

\begin{equation}
    |z_2(\lambda)| = \sqrt{2\lambda - \sqrt{\lambda(4\lambda-1)}}.
\end{equation}

It is easy to see that $|z_2| < 1$ and hence, the region of convergence is determined by the solutions of the inequality $|z_1(\lambda)| < 1$, for all $\lambda \in \Lambda$. Note that for the $\lambda > 1/4$ that satisfies $|z_1(\lambda)| = 1$ the system oscillates, given that $z_1(\lambda) \neq 1$ ($\iff \lambda \neq 0$). Moreover, it follows that $|z_1(\lambda)| < 1 \iff \lambda < 1/3$ (see \Cref{fig:complex_norm}). As a result, the system is stable if and only if every eigenvalue of $\eta^2 A A^T$ is smaller than $1/3$, or equivalently $\eta^2 \gamma^2 < 1/3$, completing the proof. 
\end{proof}

\paragraph{Optimal Learning Rate} Our analysis directly provides the learning rate that optimizes the rate of convergence. Specifically, let $\lambda_{min} = \min \{\lambda : \lambda \in \Lambda\}$ and $\lambda_{max} = \max \{ \lambda : \lambda \in \Lambda\}$. Note that the rate of convergence depends on the norm of the $z_1(\lambda)$ branch of solutions (the blue curve in \Cref{fig:complex_norm}), since $|z_1(\lambda)| \geq |z_2(\lambda)|, \forall \lambda > 0$. Indeed, the rate of convergence depends on the root with the largest norm. We consider the following two cases: 

\begin{itemize}
    \item Case $\lambda_{min} = \lambda_{max}$ (i.e. all the eigenvalues of $A A^T$ are equal): The optimal learning rate is simply $\eta_{opt} = 1/(2\sqrt{\lambda_{max}}) = 1/(2\gamma)$; for instance, see our example in \Cref{fig:learning_rate}.
    \item Case $\lambda_{min} \neq \lambda_{max}$: Consider that we begin with a very small learning rate, and we gradually increment its value. Then, the solutions $|z_1(\lambda)|$, for $\lambda \in \Lambda$, will progressively traverse the blue curve in \Cref{fig:complex_norm}. It is easy to see that the rate of convergence - i.e. the $\max \{ |z_1(\lambda)| : \lambda \in \Lambda \}$ - will decrease (the dynamics will accelerate), until we reach the learning rate $\eta_{opt}$ that satisfies $|z_1(\eta_{opt}^2 \lambda_{min})| = |z_1(\eta_{opt}^2 \lambda_{max})|$. Naturally, it follows that $\eta_{opt}^2 \lambda_{min} < 1/4$ and $\eta_{opt}^2 \lambda_{max} > 1/4$ and thus, we can obtain the optimal learning rate $\eta_{opt}$ through Equations \eqref{equation:z_1-1} and \eqref{equation:z_1-2}.
\end{itemize}

\begin{figure}[!ht]
    \centering
    \includegraphics[scale=0.45]{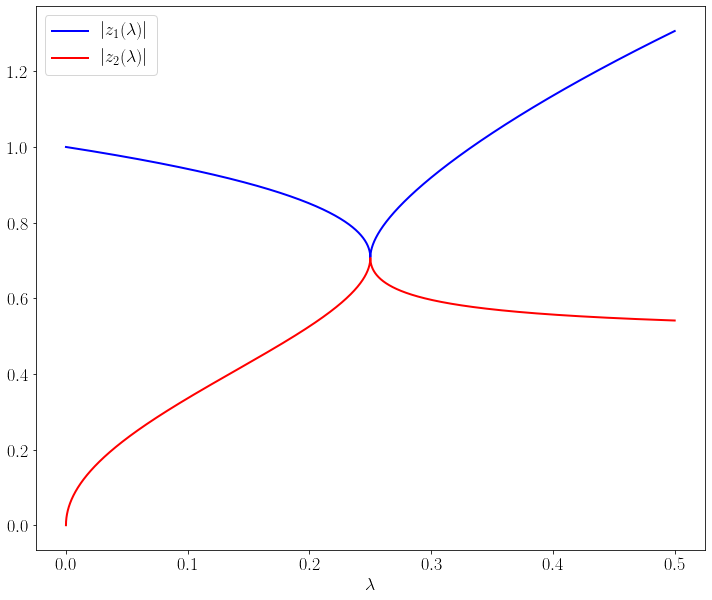}
    \caption{The norm of the solutions of \Cref{equation:quadratic_1} with respect to $\lambda$.}
    \label{fig:complex_norm}
\end{figure}

\end{document}